\documentclass[manuscript]{amsart}
\usepackage{amsmath,amsfonts,amsthm}
\usepackage{amssymb,latexsym}
\usepackage[colorlinks=true, linkcolor=blue,
urlcolor=blue]{hyperref}
\usepackage{graphicx}
\usepackage[all]{xy}
\usepackage{layout}
\chardef\bslash=`\\ 

\newtheorem{theorem}{Theorem}[section]
\newtheorem{corollary}[theorem]{Corollary}
\newtheorem{lemma}[theorem]{Lemma}
\newtheorem{proposition}[theorem]{Proposition}

\theoremstyle{definition}
\newtheorem{definition}{Definition}[section]

\newcommand{\N}{\mathbb{N}}

\newcommand{\R}{\mathbb{R}}
\newcommand{\C}{\mathbb{C}}

\newcommand{\eval}[2][\right]{\relax
  \ifx#1\right\relax \left.\fi#2#1\rvert}

\title[Discretization]{Discretization of Topological  Spaces}
\author[M. Amini and N. Golestani]{Massoud Amini and Nasser Golestani}
\address{Department of Pure Mathematics\\ Faculty of Mathematical Sciences\\
 Tarbiat Modares University\\
 Tehran\\ Iran}
\email {mamini@modares.ac.ir, n.golestani@modares.ac.ir}

\thanks{The first author was partly supported by a grant from IPM (No. 93430215).}



\begin{document}

\begin{abstract}
 There are several compactification procedures in topology,
but there is only one standard discretization, namely,
replacing the original topology with the discrete topology.
We give a notion of discretization which is dual (in categorical sense) to compactification
and give examples of discretizations.
Especially, a discretization functor from the category of $\alpha$-scattered  Stonean spaces
to the category of discrete spaces is constructed which is the converse of the
Stone-\v{C}ech compactification functor.
The interpretations of discretization in the level of algebras of functions are given.
\end{abstract}

\maketitle
\tableofcontents


\section{Introduction}

\noindent

People working with topological spaces feel relief when the underlying space is either compact or discrete, notions which are two sides of a spectrum, and in a sense "dual" to each other. It is natural to wish for a way to make topological spaces either compact or discrete. The first practice is known as compactification, and there are already a range of schemes available in general topology to make it happen. Somewhat of a surprise, there is only one standard procedure to make something discrete, namely by replacing the original topology with a discrete one. This keeps the underlying set unchanged and only changes the topology, while almost all compactification procedures change the underlying set as well as the topology.

The original motivation of the present paper was a search for a notion of ``discretization" which allows a change in the underlying set. The rationale behind the notion was that it should somehow turn out to be dual to the notion of compactification (naively, they should cancel out each other for spaces which are already compact or discrete). In the language of category theory, we expected the two procedures to be equivalences of categories, when appropriately restricted.

We expected that at least on the compactification side, most of the things we need should be already known (something which was essentially true), but we needed to restate them in the category language (Section~2). Also, as already known, some well known compactification procedures (like Bohr compactification of abelian groups) are not compactification in a strong sense (as those like one-point, or Stone-\v{C}ech compactification). We made this more explicit by defining three notions of weak, preweak, and (strong) compactification (Section 3). For a given space $X$, these correspond to a choice of unital subalgebras of $\mathrm{C}_{b}(X)$ with certain separation properties (Theorem~\ref{thrpre} and
Propositions~\ref{propcomp} and \ref{propweak}). To underline the categorical aspects of the underlying notions, we defined (weak, preweak) compactification functor (Section 4), which provides an algebraic machinery to produce various examples of compactifications (Proposition~\ref{propredu}). Parallel to the three levels of ``compactification" we introduced three notions of  weak, preweak, and (strong) discretization (Section~5). Like in the case of compactification, not every space has a discretization. A locally compact Hausdorff space
 has a discretization
if and only if
it is $\alpha$-scattered (Theorem~\ref{thrchdisc}). The duality, in the case of Stone-\v{C}ech compactification, is established by showing that
for a Tychonoff space $X$,
 $\delta(\beta(X))=X$ if and only if $X$ is discrete, and
$\beta(\delta(X))=X$ if and only if $X$ is
 Stonean and $\alpha$-scattered (Theorem~\ref{thrdual}), where
 $\delta (X)$ denotes the set of isolated points of $X$. A categorical approach to discretization is given (Section~6) and it is shown that the functor
$\delta :\mathcal{D}\to  \mathbf{Disc}$
from the category of $\alpha$-scattered  Stonean spaces
to the category of discrete spaces is the natural inverse of the
Stone-\v{C}ech compactification functor $\beta :\mathbf{Disc}\to \mathcal{D}$
(Theorem~\ref{thrcdual}). Various discretizations of a locally compact Hausdorff
space $X$ are in correspondence with certain ideals of $\mathrm{C}_{0}(X)$ (Section~7).
It is shown that  a locally compact Hausdorff space
$X$ has a discretization if and only if
$\mathrm{C}_{0}(X)$ has a closed essential ideal which is generated by
its minimal projections (Corollary~\ref{cordisc}).

\section{Preliminaries}
\noindent
In this section we establish some terminologies and review some known results on
the functor $\mathrm{C}$ and the Gelfand transform $\mathcal{G}$. We give
some modifications that we shall need in the sequel. We follow
\cite{en89} for  all topological noitions unless otherwise stated.

Let us recall the definition and some known properties of the functor
$\mathrm{C}_{b}: \mathbf{Top}_{3\frac{1}{2}}\to \mathbf{C^{*}com}_{1}$ from
the category of Tychonoff spaces with continuous maps  to
the category of commutative unital C$^{*}$-algebras with unital $*$-homomorphisms.
We refer the reader to \cite{mu90}  and \cite{ta79} for the definition and
some elementary properties of C$^{*}$-algebras.
In this note we are mostly concerned with  commutative C$^{*}$-algebras,
which are isomorphic to $\mathrm{C}_{0}(X)$
for a  locally compact Hausdorff space $X$.
Note that, the involution  on $\mathrm{C}_{0}(X)$ assigns to each $f\in \mathrm{C}_{0}(X)$
its conjugate $f^{*}$ defined by $f^{*}(x)=\overline{f(x)}$.
By a $*$-homomorphism  we mean a linear map which preserves multiplication and
involution.

The contravariant functor
$\mathrm{C}_{b}:\mathbf{Top}_{3\frac{1}{2}}\to \mathbf{C^{*}com}_{1}$ assignes to $X$ the algebra
$\mathrm{C}_{b}(X)$  of all bounded continuous functions on
$X$.
If $f:X\to Y$ is a continuous map then
$\mathrm{C}_{b}(f):\mathrm{C}_{b}(Y)\to \mathrm{C}_{b}(X)$
is defined by $\mathrm{C}_{b}(f)(g)=g\circ f$.

Let
$\mathbf{Top}_{\mathrm{cpt,2}}$ be the category of compact Hausdorff spaces with continuous maps.
The restriction of $\mathrm{C}_{b}$ to $\mathbf{Top}_{\mathrm{cpt,2}}$, denoted by
$\mathrm{C}:\mathbf{Top}_{\mathrm{cpt,2}}\to \mathbf{C^{*}com}_{1}$, is an equivalence
of categories (see \cite{ma98} for categorical definitions). In fact,
define the contravariant functor
$\mathrm{\Phi}:\mathbf{C^{*}com}_{1} \to \mathbf{Top}_{\mathrm{cpt,2}}$ by
$\mathrm{\Phi}(A)=\mathrm{\Phi}_{A}$, the character  space of a C$^{*}$-algebra $A$,
i.e., the set of non-zero homomorphisms from $A$ to $\C$.
For
a unital $*$-homomorphism $\varphi :A\to B$ in $\mathbf{C^{*}com}_{1}$ define
$\mathrm{\Phi}(\varphi): \mathrm{\Phi}_{B}\to \mathrm{\Phi}_{A}$
by $\mathrm{\Phi}(\varphi)(\xi)=\xi\circ \varphi$.
For each $X\in \mathbf{Top}_{\mathrm{cpt,2}}$ let
$\mathrm{ev}_{X}:X\to \mathrm{\Phi}_{\mathrm{C}(X)}$ be the evaluation map associated to $X$
defined by $\mathrm{ev}_{X}(x)(f)=f(x)$, for  $x\in X$ and $f\in \mathrm{C}(X)$.
For
each $A\in \mathbf{C^{*}com}_{1}$ let $\mathcal{G}_{A}: A\to \mathrm{C}(\mathrm{\Phi}_{A})$
be the Gelfand transform associated to $A$, which is defined by
$\mathcal{G}_{A}(a)(\xi)=\xi(a)$, for each $a\in A$ and
$\xi\in \mathrm{\Phi}_{A}$.
Then we have the following well-known result \cite{mu90, ta79}.
\begin{theorem}\label{thrc}
The  functor $\mathrm{C}:\mathbf{Top}_{\mathrm{cpt,2}}\to \mathbf{C^{*}com}_{1}$ is an
equivalence of categories. More precisely, considering the  functor
$\mathrm{\Phi}:\mathbf{C^{*}com}_{1} \to \mathbf{Top}_{\mathrm{cpt,2}}$ we have
$\mathcal{G}:\mathrm{id}_{\mathbf{C^{*}com}_{1}}\cong
\mathrm{C}\mathrm{\Phi}$ and
$\mathrm{ev}: \mathrm{id}_{\mathbf{Comp}_{2}}\cong \mathrm{\Phi} \mathrm{C}$.
\end{theorem}
Let us recall the following property of the functor $\mathrm{C}$. This  is part of the literature.
\begin{proposition}\label{proccpt}
Let $X$ and $Y$ be compact Hausdorff spaces and $f:X\to Y$ be a continuous map.
Consider the $*$-homomorphism $\mathrm{C}(f): \mathrm{C}(Y)\to \mathrm{C}(X)$.
Then we have:
\begin{itemize}
\item[(1)]
$f$ is injective if and only if $\mathrm{C}(f)$ is surjective;
\item[(2)]
$f$ is surjective if and only if $\mathrm{C}(f)$ is injective;
\item[(3)]
$f$ is a homeomorphism if and only if $\mathrm{C}(f)$ is bijective.
\end{itemize}
\end{proposition}
We shall need a version of Proposition~\ref{proccpt} for Tychonoff spaces in the sequel.
Recall that a family
$\mathfrak{F}\subseteq \mathrm{C}_{b}(X)$  separates points and
closed sets if for every closed set $E$ of $X$ and every $x\in X\setminus E$
there exists $f\in \mathfrak{F}$ such that $f(x)\not\in \overline{f(E)}$.
Also recall the Stone-\v{C}ech compactification functor
$\beta : \mathbf{Top}_{3\frac{1}{2}}\to \mathbf{Top}_{\mathrm{cpt,2}}$.
For each $X$ in $\mathbf{Top}_{3\frac{1}{2}}$,
$\beta X$ is the closure of the range of the homeomorphic embedding
$\beta_{X}:X\to X^{\mathrm{C}(X,\mathrm{I})}$ defined by
$\beta_{X}(x)(f)=f(x)$, $x\in X$, $f\in \mathrm{C}(X,\mathrm{I})$, where
$\mathrm{C}(X,\mathrm{I})$ is the set of continuous functions from $X$ to
$\mathrm{I}=[0,1]$.
 It is well-known that
$\mathrm{C}(\beta_{X}): \mathrm{C}(\beta X)\to \mathrm{C}_{b}(X)$ is a $*$-isomorphism.

We need the following lemma in the proof of Proposition~\ref{propcty}, which is a slight modification of
\cite[Proposition~2.6]{fo95} and is proved using a standard compactness argument. Recall that
for a topological space $X$, $\mathrm{C}_{0}(X)$ denotes the set of all complex-valued
continuous functions on $X$  vanishing at infinity.
\begin{lemma}\label{lemuni}
Let $X$ be a locally compact Hausdorff space and $\mathcal{U}$ be a uniformity
on $X$
which gives the topology of $X$. Then each function $f\in \mathrm{C}_{0}(X)$
is uniformly continuous with respect to the uniformity $\mathcal{U}$ on $X$ and the Euclidean
metric on $\mathbb{C}$.
\end{lemma}

Most parts of the next result is known (except the last part, which follows from the above lemma).

\begin{proposition}\label{propcty}
Let $X$ and $Y$ be Tychonoff spaces and $f:X\to Y$ be  continuous.
Consider the $*$-homomorphism $\mathrm{C}(f): \mathrm{C}_{b}(Y)\to \mathrm{C}_{b}(X)$.
Then:
\begin{itemize}
\item[(1)]
$f$ is injective if and only if $\mathrm{ran}(\mathrm{C}(f))$ separates the points of $X$;
\item[(2)]
$f$ has dense range if and only if $\mathrm{C}(f)$ is injective;
\item[(3)]
$f$ is a homeomorphic embedding if
$\mathrm{ran}(\mathrm{C}(f))$ separates  points and closed sets.
The converse is true when $X$ is locally compact.
\end{itemize}
\end{proposition}

\section{Compactification}
\noindent
In this section we  introduce two other notions (weak  and preweak compactifications)
which are closely related to the notion of compactification and
 have important examples in harmonic analysis and
operator algebras (including the Bohr compactification of a locally compact abelian group).

There is a one-to-one correspondence between (equivalent) compactifications
of a locally compact Hausdorff space $X$ and the unital C$^*$-subalgebras of
$\mathrm{C}_{b}(X)$ which separate the
points and closed sets (see Proposition~\ref{propcomp} below).
The question which arises naturally here is that
what   weaker notions of compactifications
 correspond to all  unital C$^*$-subalgebras of
$\mathrm{C}_{b}(X)$ or to those which only separate the
points.
These are shown to be the weak  and
preweak compactifications of $X$
(see Theorem~\ref{thrpre} and Proposition~\ref{propweak} below).

Later we use a similar
approach to define
a  notion of (weak, preweak) discretization
and obtain suitable counterparts
for   various types of discretizations of $X$
 in terms of
certain ideals in $\mathrm{C}_{0}(X)$.
\begin{definition}\label{defcomp}
Let $X$ be a topological space. Consider an ordered pair
$(f,Y)$ where $Y$ is a compact Hausdorff space and $f: X\to Y$ is a map.
\begin{itemize}
\item[(1)] $(f,Y)$ is called a \emph{compactification} of $X$ if $f$ is a homeomorphic
embedding  with dense range.
\item[(2)] $(f,Y)$ is called a \emph{weak compactification} of $X$ if $f$ is  continuous and injective with dense range.
\item[(3)] $(f,Y)$ is called a \emph{preweak compactification} of $X$ if $f$ is continuous  with dense range.
\end{itemize}
\end{definition}
The term ``weak'' in part~2 refers to the fact that the projective topology
induced by the map $f$ and the space $Y$ to $X$ is weaker than its original topology.
The Stone-\v{C}ech compactification and the one-point compactification of a Tychonoff space
are examples of a compactification. For a locally compact Hausdorff abelian topological group, its Bohr compactification is a weak compactification, which is 
not a compactification, unless the group is already compact \cite[Section~4.7]{fo95}. The same construction is available in the non abelian case \cite{ho64}, but gives only a preweak compactification.

For a semitopological  semigroup $S$, a semigroup compactification of $S$
is a pair $(\psi,X)$, where $X$ is a compact Hausdorff right topological semigroup,
 $\psi:S\to X$ is a continuous homomorphism with dense range, and the left translations
by elements of $\psi(S)$ are continuous on $X$ \cite[Definition~3.1.1]{bjm89}. Each semitoplogical compactification is a preweak compactification.

More examples could be constructed in Banach spaces: let 
$X=Y=\mathrm{B}(\mathcal{H})_{1}$  be the closed unit ball
of the space of bounded operators on a Hilbert space $H$. Give 
$X$ and $Y$ the strong and weak operator topologies, respectively, then $(id,Y)$ is
a weak compactification of $X$ \cite[Theorem~4.2.4]{mu90} which is not a compactification unless $\mathcal{H}$ is
finite dimensional \cite{ta79}. Also for a Banach space $\mathfrak{X}$, let $X, Y$ be the closed unit ball
of $\mathfrak{X}^{*}$ with the norm and weak$^{*}$ topologies, then $(id,Y)$ is a weak compactification of $X$ which is not a compactification unless $\mathfrak{X}$ is
finite dimensional \cite{co90}.
Alternatively, let  $X$ be the closed unit ball
of $\mathfrak{X}$ with the norm topology, and $Y$ be the closed unit ball
of $\mathfrak{X}^{**}$ with the weak$^{*}$ topology. Let $\tau :X\to Y$ be the restriction to $X$
of the canonical embedding of $\mathfrak{X}$ into $\mathfrak{X}^{**}$. Then $(\tau,Y)$ is a weak compactification of $X$ \cite[Proposition~4.1]{co90}) which is not a compactification unless $\mathfrak{X}$ is
finite dimensional.
If we consider $X$ with  the weak topology, then $(\tau,Y)$ is a compactification of $X$.
Moreover, $f$ is onto if and only if  $\mathfrak{X}$ is reflexive \cite{co90}.

It is known that a topological space has a compactification if and only
it is Tychonoff \cite[Theorem~3.2.6]{en89}. Also it is obvious that each topological
space has a preweak compactification (consider the singleton space).
One can ask  which topological spaces have a weak compactification.
Recall that a topological space $X$ is called \textit{completely Hausdorff} (functionally Hausdorff)
if whenever $x\neq y$ in $X$, there is a continuous
function $f:X\to [0,1]$ with $f(x)=0$ and $f(y)=1$ \cite{wi70}.
\begin{theorem}\label{thrwcom}
A topological space has a weak compactification if and only if it is
completely Hausdorff.
\end{theorem}
\begin{proof}
Let $(f,Y)$ be a weak compactification of $X$. Since $Y$ is completely Hausdorff and
$f$ is injective, it follows that $X$ is also completely Hausdorff.
Conversely suppose that $X$ is completely Hausdorff. Let $Z$ be the
Tychonoff cube $\mathrm{I}^{\mathrm{C}(X,\mathrm{I})}$
where $\mathrm{I}=[0,1]$.
Define  $F:X\to Z$ by $F(x)(f)=f(x)$ for each $x\in X$ and
$f\in \mathrm{C}(X,\mathrm{I})$. Let $Y$ be the closure of the range of $F$
in $Z$. Then $(F,Y)$ is a weak compactification of $X$.
\end{proof}

Let $X$ be locally compact Hausdorff space. Then,
each weak compactification of $X$ is a compactification,
if and only if $X$ is already compact. As an example, let $X$ be a locally compact, Hausdorff, and noncompact space.
Choose $x\in X$ and denote by  $X_{x}$ the one-point
compactification of $X\setminus \{x\}$
with $x$ as the point at infinity.
The open subsets of $X_{x}$ are the open subsets of
$X\setminus \{x\}$ and the sets of the form $X\setminus K$ where
$K$ is a compact subset of $X\setminus \{x\}$.
Since $X$ is not compact, $(\mathrm{id},X_{x})$ is a weak compactification of $X$
which is not a compactification.

There are natural examples of preweak compactifications which are not
weak compactifications.
The terminology ``preweak compactification"  in Definition~\ref{defcomp} is justified by the following  proposition.
\begin{proposition}
Let $X$  be a topological space and $(f,Y)$ be a preweak compactification of $X$. Define the equivalence relation
$R$ on $X$ by $xRy$ if{f} $f(x)=f(y)$. Let $q:X\to \frac{X}{R}$ be the canonical quotient map and endow
$\frac{X}{R}$ with the quotient topology. Then there is a unique continuous map
$\tilde{f}: \frac{X}{R}\to Y$ such that $\tilde{f}q=f$.
Moreover, $(\tilde{f},Y)$ is a weak compactification of $\frac{X}{R}$.
The space $\frac{X}{R}$ is $\mathrm{T}_{1}$ and if $f$ is a perfect mapping (i.e., closed and with the property
that $f^{-1}(y)$ is compact for each $y\in Y$),
then $\frac{X}{R}$ is  respectively, Hausdorff, regular, locally compact, or second countable, provided that
$X$ has the corresponding property.
\end{proposition}
\begin{proof}
The first statement is easy. For each $A\subseteq X$ we have
$q^{-1}(q(A))=f^{-1}(f(A))$ which shows that $q$ is a closed map. Moreover, $R[x]=f^{-1}(f(x))$
is compact, for each $x\in X$. The other statements follow from
\cite[Theorems~3.12 and 5.20]{ke75}.
\end{proof}
Let $X$ be a Tychonoff space. Let us denote by $\mathbf{Comp}(X)$,
$\mathbf{Comp}_{w}(X)$, and $\mathbf{Comp}_{pw}(X)$ the class of
 compactifications, weak compactifications, and preweak compactifications
of $X$, respectively. Then,
\[
\mathbf{Comp}(X)\subseteq \mathbf{Comp}_{w}(X)
\subseteq \mathbf{Comp}_{pw}(X).
\]
There is a natural way to define an ordering on these classes.
Let $(f,Y)$ and $(g,Z)$ be in
$\mathbf{Comp}_{pw}(X)$. We write
 $(f,Y)\geq (g,Z)$ if there is a continuous map
$h: Y\to Z$ such that $hf=g$, that is the following diagram is commutative:
\[
\xymatrix{
X \ar[r]^{f} \ar[rd]_{g} & Y\ar[d]^{h}\\
& Z \ .
}
\]
Since $f$ has dense range, such a map $h$ is unique (if it exists).
If one wants to consider the \textit{category} of (weak, preweak) compactificatins of
a topological space, the natural candidate  for a morphism from $(f,Y)$ to
$(g,Z)$ in this category would be a map $h$ as above. Note  that
the Hom-sets of this category have at most one element.

The relation $\leq$ is already  defined on $\mathbf{Comp}(X)$
 \cite{en89, ke75}, and most of its properties hold also on
$\mathbf{Comp}_{pw}(X)$. In particular, it is reflexive and transitive
on $\mathbf{Comp}_{pw}(X)$.
Two preweak compactifications $(f,Y)$ and $(g,Z)$ of $X$
are  \textit{equivalent} if there is a homeomorphism
$h:Y\to Z$ such that $hf=g$.
It is easy to see that  $(f,Y)$ and $(g,Z)$ are equivalent if and only if
 $(f,Y)\leq (g,Z)$ and $ (g,Z)\leq (f,Y)$.

If $(f,Y)$ is a preweak compactification of $X$, then
we have $w(Y)\leq \exp d(X)$ and hence $Y$ is embeddable
in the Tychonoff cube $I^{\exp d(X)}$ (see
 \cite[Theorem~3.5.3]{en89},  for the special case of compactification).
 Therefore, each weak compactification of $X$ is equivalent to a subspace
 of $I^{\exp d(X)}$. This shows that the equivalence classes  $\mathcal{C}_{pw}(X)$ of
 $\mathbf{Comp}_{pw}(X)$
  forms
 a set (rather than a proper class).
 The corresponding  equivalence classes of
 $\mathbf{Comp}_{w}(X)$ and $\mathbf{Comp}(X)$ are denoted
 by $\mathcal{C}_{w}(X)$ and $\mathcal{C}(X)$, respectively (the latter
 was introduced in \cite[Section~3.5]{en89}).
 Therefore, $\leq$ is a partial ordering
 on $\mathcal{C}_{pw}(X)$.

We   show that  there is a one-to-one correspondence  between (weak, preweak)
compactifications of $X$ and the spectrum of some certain  unital
C$^{*}$-subalgebras of $\mathrm{C}_{b}(X)$. By a C$^{*}$-subalgebras of $\mathrm{C}_{b}(X)$
(or $\mathrm{C}_{0}(X)$)
we mean a closed (in the uniform metric) subalgebra which is closed under involution.

Let $ \mathbf{Sub}_{1}(\mathrm{C}_{b}(X))$ denote the set of
all unital C$^{*}$-subalgebras of $\mathrm{C}_{b}(X)$ with the inclusion  $\subseteq$ as order. This is a complete lattice.
Let us define the map
$F : \mathbf{Comp}_{pw}(X)\to \mathbf{Sub}_{1}(\mathrm{C}_{b}(X))$ as follows.
Let $(f,Y)$ be in
$\mathbf{Comp}_{pw}(X)$.
By Definition~\ref{defcomp}, $f:X\to Y$ is continuous
with dense range. Consider the unital $*$-homomorphism
$\mathrm{C}(f):\mathrm{C}(Y)\to \mathrm{C}_{b}(X)$, as defined above, and set
$F(f,Y)=\mathrm{ran}(\mathrm{C}(f))$, which is a unital C$^{*}$-subalgebra of $\mathrm{C}_{b}(X)$.
Since $f$ has dense range,
$\mathrm{C}(f):\mathrm{C}(Y)\to F(f,Y)$ is a $*$-isomorphism
(cf.~ Proposition~\ref{propcty}). We need the following lemma whose proof is straightforward and is omitted. 

\begin{lemma}\label{lemord}
 Let  $X$ be Tychonoff space. Let $(f,Y)$ and $(g,Z)$ be  preweak compactifications of
 $X$.
 Then $(f,Y)\leq (g,Z)$ if and only if
 $F(f,Y)\subseteq F(g,Z)$.
 \end{lemma}

By the above lemma,
$F : \mathbf{Comp}_{pw}(X)\to \mathbf{Sub}_{1}(\mathrm{C}_{b}(X))$
is order-preserving. Moreover, it shows that
$(f,Y)$ and $(g,Z)$ in $\mathbf{Comp}_{pw}(X)$ are equivalent
if and only if $F(f,Y)= F(g,Z)$.
Therefore, $F$ induces an order isomorphism
from $\mathcal{C}_{pw}(X)$ into $\mathbf{Sub}_{1}(\mathrm{C}_{b}(X))$ which
is again denoted by $F$.
\begin{theorem}\label{thrpre}
Let $X$ be a Tychonoff space. Then the map
$F : \mathcal{C}_{pw}(X)\to \mathbf{Sub}_{1}(\mathrm{C}_{b}(X))$ is a surjective order isomorphism. In particular,
 $\mathcal{C}_{pw}(X)$ is a complete lattice.
\end{theorem}
\begin{proof}
By Lemma~\ref{lemord}, we need only to show that $F$ is surjective.
For each $A\in \mathbf{Sub}_{1}(\mathrm{C}_{b}(X))$, consider the ordered pair
$(f_{A}, \Phi_{A})$ where $f_{A}:X\to \Phi_{A}$ is defined by
$f_{A}(x)=\mathrm{ev}_{x}\upharpoonright_{A}$.
Since
$\Phi_{A}$ has the Gelfand topology,
$f_{A}:X\to \Phi_{A}$ is continuous.
Consider the unital $*$-homomorphism
$\varphi : \mathrm{C}(\Phi_{A})\to \mathrm{C}_{b}(X)$ defined by $\varphi(h)=\mathcal{G}^{-1}(h)$ where
$\mathcal{G}:A\to \mathrm{C}(\Phi_{A})$ is the Gelfand transformation on $A$.
We have
$\varphi=\mathrm{C}(f_{A})$. Since
$\varphi$ is injective,  $f_{A}$ has dense range, by Proposition~\ref{propcty}. This
shows that $(f_{A}, \Phi_{A})$ is in $\mathbf{Comp}_{pw}(X)$ and
$F(f_{A}, \Phi_{A})=A$.
\end{proof}

The minimal element of the lattice
 $\mathcal{C}_{pw}(X)$ is $F^{-1}(\mathbb{C})$ which is the equivalence class of
 the preweak compactification $(f, Y)$ of $X$ with $Y$ is a singleton and $f$
  the constant function.
 Also, the maximal element is $F^{-1}(\mathrm{C}_{b}(X))$ which is
 the equivalence class of $(\beta_{X},\beta X)$.

The map $F$
can be used to distinguish the compactifications from (pre)weak compactifications. The next result follows from Proposition~\ref{propcty}.

\begin{proposition}\label{prochar}
Let $X$ be a Tychonoff space and $(f,Y)$
be a preweak compactification of $X$. We have:
\begin{itemize}
\item[(1)]
 $(f,Y)$ is a weak compactification of $X$ if and only if
 $F(f,Y)$ separates the points of $X$;
\item[(2)]
$(f,Y)$ is a  compactification of $X$ if
 $F(f,Y)$ separates  points and closed sets.
 The converse is true when $X$ is locally compact.
\end{itemize}
\end{proposition}

Using the above proposition, one can give more examples of weak and preweak  compactifications.
Let $X$ be Tychonoff space and $A$ be a unital C$^*$-subalgebra of $\mathrm{C}_{b}(X)$.
Consider the $*$-homomorphism
$\varphi : \mathrm{C}(\Phi_{A})\to \mathrm{C}_{b}(X)$ defined by $\varphi(h)=\mathcal{G}^{-1}(h)$ where
$\mathcal{G}:A\to \mathrm{C}(\Phi_{A})$ is the Gelfand transformation on $A$. Then $(\Phi(\varphi), \Phi_{A})$
is a preweak compactification of $X$. It is a   weak compactification if and only if $A$
separates the points of $X$. If $X$ is locally compact, then $(\Phi(\varphi), \Phi_{A})$
is a compactification of $X$ if and only if $A$ separates points and closed sets
 (cf.~Proposition~\ref{prochar}).

Note that, for a Tychonoff space $X$, the equivalence classes of $(\beta_{X},\beta X)$
in $\mathcal{C}(X)$ and $\mathcal{C}_{w}(X)$ are the maximal
elements of these partially ordered sets.
For the minimal element of $\mathcal{C}(X)$ we have
the following result.

\begin{theorem}[\cite{en89}, Theorem~3.5.12]
Let $X$ be a Tychonoff space. Then $\mathcal{C}(X)$ has a minimal
element if and only if $X$ is locally compact.
If $X$ is  locally compact and noncompact,
then the equivalence class of the one-point compactification of $X$
is the minimal element of $\mathcal{C}(X)$.
\end{theorem}

We  need the following
lemma  in the proof of Proposition~\ref{propcomp}; it
guarantees the existence of
a minimal element in $\mathcal{C}(X)$. The proof of one direction is straightforward. The other direction follows from the above proposition.  
\begin{lemma}\label{lemmin}
Let $(f,Y)$ be a preweak compactificaion of a locally compact space $X$. Then
$(f,Y)$ is a compactification of $X$ if and only if
 $\mathrm{C}_{0}(X)\subseteq F_{X}(f,Y)\subseteq \mathrm{C}_{b}(X)$.
\end{lemma}

The following result is part of the literature
\cite[Theorems~3.5.9 and 3.5.12]{en89}, proved using topological methods.
Using Theorem~\ref{thrpre}, Proposition~\ref{prochar}(2) and 
Lemma~\ref{lemmin} one could give an algebraic proof.

\begin{proposition}\label{propcomp}
Let $X$ be a locally compact Hausdorff space.
Then the map $F$ is an order isomorphism from
$\mathcal{C}(X)$ onto the set of all unital C$^*$-subalgebras of
$\mathbf{Sub}_{1}(\mathrm{C}_{b}(X))$  which
separate  points and closed sets.
In particular, $\mathcal{C}(X)$ is a complete lattice.
\end{proposition}

One may guess that for a locally compact, Hausdorff, and noncompact space $X$,
the one-point compactification of $X$ is the minimum element of
$\mathcal{C}_{w}(X)$. But, this is not the case.

\begin{lemma}\label{lemnomin}
Let $X$ be a locally compact, Hausdorff, and noncompact space.
Then $\mathcal{C}_{w}(X)$ has no minimum element.
\end{lemma}
\begin{proof}
Suppose that $\mathcal{C}_{w}(X)$ has a minimum  element
$(f,Y)$.
Choose $x,y\in X$ with $x\neq y$. Consider
the weak compactifications $(\mathrm{id}, X_{x})$ and
$(\mathrm{id}, X_{y})$ of $X$ as described above.
Since $(f,Y)\leq (\mathrm{id}, X_{x})$, there is an injective continuous map
$h:X_{x}\to Y$
with dense range
 such that $h\circ \mathrm{id}_{X}=f$. Thus $h=f$ and
$f:X_{x}\to Y$ is a homeomorphism. Similarly, $f:X_{y}\to Y$ is  a homeomorphism.
Hence $\mathrm{id}:X_{x}\to X_{y}$ is a homeomorphism.

Since $X$ is locally compact and Hausdorff, there is an open set $U$
with compact closure  in $X$ such that
$x\in U\not\ni y$. Then $X_{y}\setminus U$ is compact in $X_{y}$
and so in $X_{x}$. Since $X_{y}\setminus U\subseteq X\setminus \{x\}$,
$X\setminus U$ is also compact in $X$. Finally, $X=(X\setminus U) \cup\overline{U}$ is compact.
\end{proof}

The next result follows from Theorem~\ref{thrpre}, Proposition~\ref{prochar},  and
Lemma~\ref{lemnomin}.
\begin{proposition}\label{propweak}
Let $X$ be a locally compact Hausdorff space.
Then the map $F$ is an order isomorphism from
$\mathcal{C}_{w}(X)$ onto the set of all unital C$^*$-subalgebras of
$\mathbf{Sub}_{1}(\mathrm{C}_{b}(X))$  which
separate   the points of $X$.
In particular, $\mathcal{C}_{w}(X)$ has the least upper bound property.
Moreover, it has a minimum element if and only if $X$ is compact.
\end{proposition}

If $X$ is a compact Hausdorff space, then
$\mathcal{C}_{w}(X)$ has only one element, that is, the equivalence class of
$(\mathrm{id},X)$. Moreover, $\mathcal{C}_{pw}(X)$ is
the set of all equivalence classes of Hausdorff quotients of $X$, i.e.,
$(\pi_{R}, {X}/{R})$ where $R$ is an equivalence relation on $X$
such that ${X}/{R}$ is Hausdorff and
$\pi_{R}:X\to {X}/{R}$ is the quotient map. Note that, since   $X$
 is compact so is ${X}/{R}$, and ${X}/{R}$ is Hausdorff
if and only if $R$ is closed in $X\times X$.

The following proposition identifies the position of $\mathcal{C}(X)$ and
$\mathcal{C}_{w}(X)$ in $\mathcal{C}_{pw}(X)$ from lattice theory point of view (see
\cite{bs81} for definition of a filter).

\begin{proposition}
Let $X$ be a Tychonoff space and
$(f,Y)$ and $(g,Z)$ be a pair of preweak compactifications of
$X$ such that $(f,Y)\geq (g,Z)$. Then:
\begin{itemize}
\item[(1)]
if  $(g,Z)$ is weak compactification of $X$, then so is $(f,Y)$;
\item[(2)]
if  $(g,Z)$ is  compactification of $X$, then so is $(f,Y)$.
\end{itemize}
Moreover, $\mathcal{C}(X)$ is a filter of $\mathcal{C}_{pw}(X)$, and
$\mathcal{C}_{w}(X)$ is not a filter of $\mathcal{C}_{pw}(X)$ unless
$X$ is  compact.
\end{proposition}
\begin{proof}
Since $(f,Y)\geq (g,Z)$, there is a continuous map
$h:Y\to Z$ such that $hf=g$. If
 $(g,Z)$ is weak compactification of $X$, then $g$ is one-to-one, and
 so is $f$. Thus $(f,Y)$ is also a weak compactification of $X$. This proves part (1).
 The proof of part (2) is similar.

 That $\mathcal{C}(X)$ is a filter of $\mathcal{C}_{pw}(X)$
 follows from  part (2) and the fact that $\mathcal{C}(X)$ is a sublattice of $\mathcal{C}_{pw}(X)$.
Let $X$ be a locally compact, Hausdorff, and noncompact space. Choose $x,y\in X$
 with $x\neq y$. The proof of Lemma~\ref{lemnomin} shows that
 there is no $(f,Y)$ in $\mathcal{C}_{w}(X)$ such that
 $(f,Y)\leq (\mathrm{id}, X_{x})$ and $(f,Y)\leq (\mathrm{id}, X_{y})$.
 Therefore, $\mathcal{C}_{w}(X)$ is not a filter of $\mathcal{C}_{pw}(X)$.
\end{proof}

\section{Compactification Functor}
\noindent
In this section we define a (weak, preweak) compactification functor.
The most important example is the Stone-\v{C}ech compactification functor.
We introduce an algebraic machinery
which enables us to give more examples. A similar approach is given in Section~6 for discretization.
\begin{definition}\label{defcofu}
A  functor
$F: \mathcal{C}\to  \mathbf{Top}_{\mathrm{cpt,2}}$ where
$\mathcal{C}$ is a subcategory of  $\mathbf{Top}_{3\frac{1}{2}}$
is called a \emph{preweak compactification functor}
if
there is a correspondence $\tau$ which assigns to each $X$ in $\mathcal{C}$
 a continuous map $\tau_{X}:X\to F(X)$ with dense range  such that for each morphism
$f:X\to Y$ in $\mathcal{C}$
the following diagram commutes:
\[
\xymatrix{
X \ar[d]_{f}\ar[r]^{\tau_{X}} & F(X) \ar[d]^{F(f)} \\
Y\ar[r]_{\tau_{Y}}  & F(Y)\ .
}
\]

A preweak compactification functor
$F$ is called a \emph{weak compactification functor} (respectively, \textit{compactification functor})
if $\tau_{X}$ is  injective (respectively, homeomorphic embedding).
\end{definition}
When the domain category is the whole $\mathbf{Top}_{3\frac{1}{2}}$,
the Stone-\v{C}ech compactification functor
$\beta : \mathbf{Top}_{3\frac{1}{2}}\to \mathbf{Top}_{\mathrm{cpt,2}}$
 is the only weak compactification functor, as shown in the following theorem.
 
\begin{theorem}
Let $F: \mathbf{Top}_{3\frac{1}{2}}\to \mathbf{Top}_{\mathrm{cpt,2}}$
be a weak compactification functor. Then
$F$ is naturally isomorphic to the Stone-\v{C}ech compactification functor.
\end{theorem}
\begin{proof}
Let $\tau$ be the correspondence associated to $F$ as in Definition~\ref{defcofu}.
Let $X$ be a Tychonoff space.
Since $\tau_{X}:  X\to F( X)$ is continuous with dense range,
there is a unique continuous map
$\sigma_{X}:\beta X\to F(X)$ such that $\sigma_{X}\beta_{X}=\tau_{X}$.
It is obvious that $\sigma_{X}$ is onto. We claim that $\sigma_{X}$ is injective.
In fact, first note that the following diagram commutes:
\[
\xymatrix{
X \ar[d]_{\beta_{X}}\ar[r]^{\tau_{X}} & F(X) \ar[d]^{F(\beta_{X})} \\
\beta X\ar[r]_{\tau_{\beta X}}  & F(\beta X)\ .
}
\]
Hence  $F(\beta_{X})\sigma_{X}\beta_{X}=F(\beta_{X})\tau_{X}=\tau_{\beta_{X}}\beta_{X}$.
Since $\beta_{X}$ has dense range,
$F(\beta_{X})\sigma_{X}=\tau_{\beta_{X}}$.
On the other hand, $\tau_{\beta X}:\beta X\to   F(\beta X)$ is a homeomorphism.
Thus $\sigma_{X}$ is injective and the claim is proved.

Now let
$f:X\to Y$ be a morphism in $ \mathbf{Top}_{3\frac{1}{2}}$.
We have
\[
F(f)\sigma_{X}\beta_{X}=F(f)\tau_{X}=\tau_{Y}f=
\sigma_{Y}\beta_{Y}f=\sigma_{Y}(\beta f)\beta_{X}.
\]
Thus $F(f)\sigma_{X}=\sigma_{Y}\beta f$ and the following diagram commutes:
\[
\xymatrix{
\beta X \ar[d]_{\beta f}\ar[r]^{\sigma_{X}} & F(X) \ar[d]^{F(f)} \\
\beta Y\ar[r]_{\sigma_{Y}}  & F(Y)\ .
}
\]
This shows that $\sigma:\beta\cong F$, i.e., $F$ is naturally isomorphic to $\beta$.
\end{proof}

Let $\mathcal{C}$ be the category of non-compact locally compact Hausdorff spaces with
proper maps
(continuous maps which lift back compact sets to compact sets) as morphisms.
Let $A: \mathcal{C}\to \mathbf{Top}_{\mathrm{cpt,2}}$ be the one-point compactification
functor. The functor $A$ acts on morphisms by extension.
Let each $\tau_{X}:X\hookrightarrow A(X)$ be the injection map.
It is easy to see that
$A: \mathcal{C}\to \mathbf{Top}_{\mathrm{cpt,2}}$ is a compactification functor.

As another example, let $\mathbf{TG}_{\mathrm{lc,2}}$ and $\mathbf{TG}_{\mathrm{cpt,2}}$
denote the categories of locally compact Hausdorff  groups and
compact Housdorff groups, respectively, both with
continuous homomorphisms. Let
$\mathrm{b}: \mathbf{TG}_{\mathrm{lc,2}}\to\mathbf{TG}_{\mathrm{cpt,2}}$
be the  Bohr compactification functor \cite{fo95}. It is
easy to see that $\mathrm{b}$ is a weak compactification functor.

We shall give the structure of all (preweak, weak) compactification functors
from the function algebra point of view
using the following
observation
(see Proposition~\ref{propredu} below).

Let $\mathcal{C}$ be a subcategory of  $\mathbf{Top}_{3\frac{1}{2}}$ and
$G:\mathcal{C}\to \mathbf{C^{*}com}_{1}$ be a contravariant functor such that
for each $X\in \mathcal{C}$, $G(X)$ is a unital C$^*$-subalgebra of $\mathrm{C}_{b}(X)$,
and for each morphism $f:X\to Y$ in $\mathcal{C}$, $G(f)$ is the restriction
of $\mathrm{C}(f)$ to $G(Y)$.
Then $G$ induces a preweak compactification functor
$\widetilde{G}:\mathcal{C}\to \mathbf{Top}_{\mathrm{cpt,2}}$ as follows.
Moreover, if each $G(X)$ separates the points of $X$, then $\widetilde{G}$ will be a weak compactification functor,
and if each $G(X)$ separates  points and closed sets then $\widetilde{G}$ will be a  compactification functor.

Set $\widetilde{G}=\mathrm{\Phi}\circ G$, where $\mathrm{\Phi}$ is as in
 the paragraph preceding Theorem~\ref{thrc}.
Then $\widetilde{G}$ is a covariant functor from $\mathcal{C}$ to
$\mathbf{Top}_{\mathrm{cpt,2}}$.
 For each $X\in \mathcal{C}$ define
$\tau_{X}:X\to \widetilde{G}(X)$ by $\tau_{X}(x)=\mathrm{ev}_{x}$, the evaluation at $x$ from
$G(X)$ to $\mathbb{C}$.
It is easy to see that $\widetilde{G}$ together with the correspondence $\tau$ form a
preweak compactification functor.

The functor $\widetilde{G}$ in the above example  is called
the \emph{induced functor} of $G$. The following theorem says that each
(preweak, weak) compactification
functor is (naturally isomorphic to) the induced functor of some functor.

\begin{proposition}\label{propredu}
Let $F: \mathcal{C}\to  \mathbf{Top}_{\mathrm{cpt,2}}$ be a preweak
compactification functor where
$\mathcal{C}$ is a subcategory of  $\mathbf{Top}_{3\frac{1}{2}}$.
Then there is a unique (up to natural isomorphism)  contravariant functor
$G:\mathcal{C}\to \mathbf{C^{*}com}_{1}$ satisfying properties listed above, such that
$F$ is natural isomorphic to the induced functor of $G$.
Moreover, $F$ is a weak compactification functor
if and only if for each $X$ in $\mathcal{C}$,  $G(X)$ separates the points of $X$.
If each object of $\mathcal{C}$ is locally compact, then
$F$ is a  compactification functor
if and only if for each $X$ in $\mathcal{C}$,  $G(X)$
separates  points and closed sets.
\end{proposition}
\begin{proof}
Define the functor $G:\mathcal{C}\to \mathbf{C^{*}com}_{1}$ as follows.
For each $X$ in $\mathcal{C}$, set $G(X)=F_{X}(\tau_{X},F(X))$, where $F_{X}$ is the
map associated to $X$ as defined in the paragraph preceding Lemma~\ref{lemord}.
 Then
$G(X)$ is a unital C$^{*}$-subalgebra of $\mathrm{C}_{b}(X)$.
Let $f:X\to Y$ be a morphism in $\mathcal{C}$.
By Definition~\ref{defcomp}, the following diagrams commute.

\[
\xymatrix{
X \ar[d]_{f}\ar[r]^{\tau_{X}} & F(X) \ar[d]^{F(f)} \\
Y\ar[r]_{\tau_{Y}}  & F(Y)\ ,
}
\ \ \ \  \   \ \ \ \ \ \ \
\xymatrix{
\mathrm{C}_{b}(X)
& \mathrm{C}(F(X))  \ar[l]_{\mathrm{C}(\tau_{X})}  \\
\mathrm{C}_{b}(Y)\ar[u]^{\mathrm{C}(f)}
& \mathrm{C}(F(Y)) \ar[u]_{\mathrm{C}(F(f))} \ar[l]^{\mathrm{C}(\tau_{Y})} .
}
\]

Therefore,
$\mathrm{C}(f)\mathrm{C}(\tau_{Y})=
\mathrm{C}(\tau_{X})\mathrm{C}(F(f))$, and hence
$\mathrm{C}(f)(G(Y))\subseteq G(X)$.
Now define $G(f):G(Y)\to G(X)$ to be the restriction of
$\mathrm{C}(f)$ to $G(Y)$. Since $\mathrm{C}$ is a contravariant functor, so is $G$.
Note that $\mathrm{C}(\tau_{X}): \mathrm{C}(F(X))\to G(X)$ is a $*$-isomorphism.
Thus we have the following commutative diagram:

\[
\xymatrix{
\mathrm{\Phi}(G(X))\ar[d]_{\mathrm{\Phi}(G(f))}
\ar[r]^(.44){\mathrm{\Phi}(\mathrm{C}(\tau_{X}))} &
\mathrm{\Phi}(\mathrm{C}(F(X)))
\ar[d]^{\mathrm{\Phi}(\mathrm{C}(F(f)))} &
F(X)\ar[l]_(.35){\mathrm{ev}_{X}}\ar[d] ^{F(f)}
\\
\mathrm{\Phi}(G(X)) \ar[r]_(.45){\mathrm{\Phi}(\mathrm{C}(\tau_{Y}))}&
\mathrm{\Phi}(\mathrm{C}(F(Y))) &
F(Y)\ar[l]^(.35){\mathrm{ev}_{Y}}
}
\]

\noindent
with horizontal arrows homeomophisms. Define
 $\sigma_{X}:F(X)\to \mathrm{\Phi}(G(X))$ by
 $\sigma_{X}=\mathrm{\Phi}(\mathrm{C}(\tau_{X}))^{-1}\circ \mathrm{ev}_{X}$.
 Then $\sigma_{X}$ is a homeomorphism. Therefore,
 $\sigma$ is a natural isomorphism from $F$ to $\widetilde{G}= \mathrm{\Phi}\circ G$.
For uniqueness of $G$, by Theorem~\ref{thrc} we have
\[
\mathrm{C}\circ F\cong \mathrm{C} \circ \mathrm{\Phi}\circ G\cong
\mathrm{id}_{\mathbf{C^{*}com}_{1}} \circ G=G.
\]
Finally, the last statement follows from Proposition~\ref{prochar}.
\end{proof}

The inducing functor of the Stone-\v{C}eck compactification functor
is $G(X)=\mathrm{C}_{b}(X)$ and
 that of the one-point compactification is $G(X)=\mathrm{C}_{0}(X)+\mathbb{C}1$.

Let $\mathbf{TG}_{\mathrm{lc,2}}$ be as above.
For each  $G$ in $\mathbf{TG}_{\mathrm{lc,2}}$, let $AP(G)$, $WAP(G)$,
 and $\mathcal{UC}(G)$
denote the subsets of $\mathrm{C}_{b}(G)$ containing
almost periodic functions, weakly almost periodic functions, and uniformly continuous functions, respectively.
Each of the maps $AP$, $WAP$,  and
$\mathcal{UC}$ gives rise to a  functor from $\mathbf{TG}_{\mathrm{lc,2}}$
to $\mathbf{C^{*}com}_{1}$. Therefore,  we obtain weak compactification functors
$\widetilde{AP}$ and $\widetilde{WAP}$, and compactification functor
$\widetilde{\mathcal{UC}}$
from $\mathbf{TG}_{\mathrm{lc,2}}$ to
$\mathbf{TG}_{\mathrm{cpt,2}}$.

\section{Discretization}
\noindent
There are two main approaches to correspond  to a topological space a discrete space (or a space
which is close to being discrete in some way, e.g., zero-dimensional).
In the first approach, one tries to add some new open sets to the original topology.
For example, adding all possible open sets to the topology of a space $X$ gives
$X_{d}$. The intermediate cases are of interests and give interesting examples in topology
(e.g., the Michael line \cite[Example~5.1.32]{en89}). The idea of ``partial discretization"
follows this approach \cite{bcm00}.

This approach seems to have no suitable counterpart
in function algebras, since partial discretizations
of a locally compact Hausdorff space is not generally locally compact.
For example, consider the ``usual" Hausdorff nonregular  space \cite[Example~1.5.6]{en89} which is
a partial $\omega_{1}$-discretization of the unit interval $[0,1]$ \cite{bcm00}.
Also, the Michael line which is constructed from the real line is not locally compact.
However, a central object in function algebras
(and C$^*$-algebras) is
 $\mathrm{C}_{0}(X)$, and  $X$ is the spectrum of
$\mathrm{C}_{0}(X)$, only when it is locally compact and Hausdorff (see Theorem~\ref{thrc}).

Another approach to  discretization (which is adopted here) is to consider
the ``discrete" subspaces of a topological space. In some sense, this is dual to the notion of
 compactification. For example, for any discrete space $X$, our definition implies that
$X$ is a discretization of $\beta X$ (in fact, it is the maximal discretization of $\beta X$,
see the last paragraph of this section).
\begin{definition}\label{defdisc}
Let $X$ be a topological space. Consider an ordered pair
$(f,Y)$ where $Y$ is a discrete space and $f: Y\to X$ is a continuous map.
\begin{itemize}
\item[(1)] $(f,Y)$ is called a \emph{discretization} of $X$ if  $f$ is a homeomorphic
embedding with dense range.
\item[(2)] $(f,Y)$ is called a \emph{weak discretization} of $X$ if $f$ is  a
homeomorphic embedding with open range.
\item[(3)] $(f,Y)$ is called a \emph{preweak discretization} of $X$ if $f$ is injective.
\end{itemize}
\end{definition}

It is obvious that each discretization of a submaximal space (i.e., each dense subset of $X$ is open \cite{ac95}) is a weak discretization.
Also each discretization of a $\mathrm{T}_{1}$ space is a  weak discretization
(see Theorem~\ref{thrt1} below).
Each weak discretization  of a topological space  is a preweak discretization.
If $(f,Y)$ is a weak discretization of $X$, then $(f,Y)$ is a discretization of any subspace
$Z$ of $X$ with  $f(Y)\subseteq Z\subseteq \overline{f(Y)}$.
Also, the property of having a discretization is hereditary in taking open
subsets. Note that if $X$ is already discrete, then the only  discretization
of $X$ (up to homeomorphism) is $X$ itself.

In Definition~\ref{defdisc}, the condition on the range of
the map in definition of weak discretization  enables us to describe weak discretization of a locally compact
Hausdorff space $X$ in terms of certain ideals of $\mathrm{C}_{0}(X)$ (see Section~7).

Let $X$ be a topological space and $X_{d}$ be the set $X$ with the discrete topology.
Then $(\mathrm{id}_{X}, X_{d})$ is a preweak discretization of $X$
which is not a weak discretization unless $X$ is already discrete. Alternatively, 
let $Z$ be any topological space and $Y$ be a discrete space. Suppose (without loss of generality) that
$Z\cap Y=\emptyset$. Set $X=Z\cup Y$ and define the topology of $X$ as follows. $U\subseteq X$
is open if $U\subseteq Y$ or there is an open set $V$ of $Z$ such that  $U=V\cup Y$.
Let $f:Y\to X$ be the injection map. Hence $Z$ is a subspace of $X$ and $Y$ is an open
discrete subspace of $X$.
Then $(f,Y)$ is both a discretization and a weak discretization of $X$.

As an example of discretization, let $Z$ be a discrete space. Then $(\beta_{Z},  Z)$ is a discretization of $\beta Z$.
An important special case is obtained by setting $Z=\mathbb{N}$. Thus $\mathbb{N}$
is  a discretization of $\beta \mathbb{N}$.
Also $\mathbb{N}$ is a discretization of $\{0\}\cup\{\frac{1}{n} \mid n\in \N\}$ which
is the one-point compactification of $\N$ with the relative topology from $\mathbb{R}$.

Each topological space has a preweak discretization.
Also the empty set is a (trivial) weak discretization of any topological space. But it is not
the case that each topological space has a  discretization.

\begin{theorem}\label{thrt1}
Let $X$ be a $\mathrm{T}_{1}$ space. Then for each discretization $(f,Y)$
of $X$. Then  $\delta(X)=f(Y)$. In particular, we have:
\begin{itemize}
\item[(1)]
each discretization of $X$ is a weak discretization of it;
\item[(2)]
if moreover $X$ has no
isolated points, then $X$ has no discretization and no nonempty weak discretization.
\end{itemize}
\end{theorem}
\begin{proof}
It is obvious that for each discretization $(f,Y)$
of $X$,  $\delta(X)\subseteq f(Y)$.
We show the reverse inclusion for a  $\mathrm{T}_{1}$ space $X$.
Let $z\in f(Y)$. Thus there is an open subset $U$ of $X$ such that
$U\cap f(Y)=\{z\}$. Taking closure (in $X$) from both sides we get
$\overline{U}=\overline{\{z\}}=\{z\}$. Therefore, $U=\{z\}$ is open in $X$ and hence
$f(Y)\subseteq \delta X$.

Part (1) follows from the first part and the fact that $\delta X$ is open in $X$.
Finally, part~(2) follows from the first part of the statement.
\end{proof}

The space $\mathbb{R}$ with the usual topology and its subspace $[0, 1]$
have no (weak) discretization (by Theorem~\ref{thrt1}).
On the other hand, $\mathbb{N}$ is a
preweak discretization
of these spaces. Alternatively, let $X$ be a set with the anti-discrete topology, i.e., the only open subsets are
$\emptyset$ and $X$. Then for each $x\in X$, $(j,\{x\})$ is a discretization of $X$
where $j$ is the inclusion map from $\{x\}$ to $X$. If $| X|\geq 2$, then
$\delta X=\emptyset$ (cf.~Theorem~\ref{thrt1}).

Let $\{X_{i}\mid i\in I\}$ be an infinite family of $\mathrm{T}_{1}$  spaces
such that $|X_{i}|\geq 2$, $i\in I$. Then $\prod_{i\in I}X_{i}$ with the Tychonoff
topology has no discretization. This is because this space has no isolated points.

The following which is an immediate consequence of Definitions~\ref{defcomp} and~\ref{defdisc},
justifies both definitions.

\begin{theorem}\label{thrdiscom}
Let $X$ be a discrete space, $Y$ be a compact Hausdorff space, and $f:X\to Y$ be a (continuous) map.
Then $(f,Y)$ is a  compactification
of $X$ if and only if $(f,X)$ is a  discretization of $Y$.
\end{theorem}
It is natural to ask  which topological spaces have a discretization.
To answer this question first let us recall some topological notions. For a space $X=(X,\tau)$,
the $\alpha$-space of $X$ is
$X^{\alpha}=(X,\tau^{\alpha})$ where
\[
\tau^{\alpha}=\{U-N: U\in \tau\ \text{and}\ N\ \text{is nowhere dense in}\  X\}.
\]
The space $X$ is called an $\alpha$-\emph{space} if $X=X^{\alpha}$.
A space $X$ is called \emph{scattered} if every nonempty subspace has an isolated point.
$X$   is an $\alpha$-\emph{scattered} space if $X^{\alpha}$ is scattered.
We denote by $\delta X$  the set of isolated points of $X$.
It is known that
a space $X$ is scattered if and only if $X$ is $\alpha$-scattered and every nonempty nowhere
dense subspace of $X$ has an isolated point \cite[Theorem~2.1]{ro98}.
The following  gives a characterization of $\alpha$-scattered spaces \cite{ro98}.

\begin{theorem}[D.\,A. Rose]\label{thrrose}
For a topological space $X$, the following are equivalent:
\begin{itemize}
\item[(1)]
$X$ is $\alpha$-scattered;
\item[(2)] every somewhere dense subspace of $X$ has an isolated point;
\item[(3)] $\delta X$ is dense in $X$.
\end{itemize}
\end{theorem}
The following theorem gives a characterization of having a discretization.
\begin{theorem}\label{thrchdisc}
Let $X$ be a $\mathrm{T}_{1}$ space. Then the following   are equivalent:
\begin{itemize}
\item[(1)]
$X$ has a discretization;
\item[(2)]
$X$ is $\alpha$-scattered;
\item[(3)]
$\delta X$  is dense in $X$.
\end{itemize}
\end{theorem}
\begin{proof}
The equivalence of (2) and (3) follows from Theorem~\ref{thrrose}.
Also the implication $(3)\Rightarrow (1)$ is obvious.
Finally, the implication $(1)\Rightarrow (3)$ follows from Theorem~\ref{thrt1}.
\end{proof}
We will give another characterization for discretization in terms of function algebras in
 Section~7.

Let us classify all (weak, preweak) discretizations of a topological space.
Let $X$ be a topological space.
Let us denote by $\mathbf{Disc}(X)$, $\mathbf{Disc}_{w}(X)$, and $\mathbf{Disc}_{pw}(X)$
the class of all discretizations,
weak  discretizations, and preweak discretizations of $X$, respectively.
 There is a natural ordering on theses classes (similar to compactifications).
 In fact, let $(f,Y)$ and $(g,Z)$ be a pair of preweak discretizations
 of $X$. The relation $(f,Y)\leq (g,Z)$  means that there is  a map
 $h:Y\to Z$ such that $gh=f$, i.e., the  diagram

 \[
\xymatrix{
Y \ar[r]^{f} \ar[d]_{h} & X\\
 Z \ar[ur]_{g}&
}
\]
commutes. Note that  $h$ with this property is unique and injective. Also
$(f,Y)$ and $(g,Z)$ are called \emph{equivalent} if
there is a bijective map $h:Y\to Z$ such that $gh=f$.
It is easy to see that $(f,Y)$ and $(g,Z)$ are equivalent if and only
if $(f,Y)\leq (g,Z)$ and $(g,Z) \leq (f,Y)$. Also
the relation $\leq $ is reflexive and transitive.

Let $X$ be a topological space. By Definition~\ref{defdisc},
each discretization and each weak discretization of $X$
is embeddable in $X$, hence  are equivalent to
some subspace of $X$.
This shows that the equivalence classes of
 $\mathbf{Disc}(X)$ under the equivalence relation  above, forms
 a set,   denoted it by
 $\mathcal{D}(X)$. The corresponding quotient set of
 $\mathbf{Disc}_{w}(X)$ and $\mathbf{Disc}_{pw}(X)$ are denoted
 by $\mathcal{D}_{w}(X)$ and $\mathcal{D}_{pw}(X)$, respectively.
 Therefore, the ordering $\leq$ induces a partial ordering
 (again denoted by $\leq$)
 on  $\mathcal{D}(X)$,  $\mathcal{D}_{w}(X)$, and $\mathcal{D}_{pw}(X)$.
 Let us summarize theses results.
 \begin{proposition}
 Let $X$ be a topological space.
 Then the ordering $\leq$, as defined above, is a partial ordering on
 $\mathcal{D}(X)$,  $\mathcal{D}_{w}(X)$, and $\mathcal{D}_{pw}(X)$.
Each element of $\mathcal{D}_{pw}(X)$ is equivalent to a preweak discretization
of the form $(j,Y_{d})$ where $Y$ is a subset of $X$
and $j$ is the injection map from $Y_{d}$ to $X$. The subset $Y$ with this property
is unique. Therefore, $\mathcal{D}_{pw}(X)$ is a complete lattice
having $ X_{d}$ as the maximal element and $\emptyset$ as the minimal element.
 \end{proposition}
\begin{proof}
The first part of the statement is obvious as described above. The second part follows directly from
Definition~\ref{defdisc}.
\end{proof}
 It is easy to describe all elements of  $\mathcal{D}_{w}(X)$ for an arbitrary topological space
 and  those of $\mathcal{D}(X)$ for a  locally compact  Hausdorff space.
 \begin{proposition}\label{propwdisc}
Let $X$ be a topological space. Then there is a one-to-one
correspondence between subspaces of $\delta X$ and  members of $\mathcal{D}_{w}(X)$
which assigns to each $Y\subseteq\delta X$ the weak discretization $(j,Y)$ of $X$, where $j$ is
the injection map from $Y$ to $X$. Therefore, $\mathcal{D}_{w}(X)$ is a complete
lattice having $\delta X$ as the maximal element and $\emptyset$ as the minimal element.
If $X$ is $\mathrm{T}_{1}$, then $X$ has a discretization if and only if
$\delta X$ is dense in $X$. In this case, the only element of $\mathcal{D}(X)$ is
the equivalence class of $(j,\delta X)$.
\end{proposition}
\begin{proof}
The first part  follows from the fact that
if $(f,Y)$ is a weak discretization of $X$, then each element of $f(Y)$ is an isolated point of $X$
(since $f(Y)$ is open in $X$). The last part follows from
Theorem~\ref{thrchdisc}.
\end{proof}

If a topological space is not $\mathrm{T}_{1}$, then $\mathcal{D}(X)$
may have many elements.

\begin{theorem}
Let $(f,Y)$ be a discretization of a topological space $X$. Then we have
$d(X)=|Y|$, where $d(X)$ is the density of $X$.
\end{theorem}
\begin{proof}
Since $f(Y)$ is dense in $X$ and $f$ is injective, $d(X)\leq |Y|$. For each $y\in Y$ there
is an open subset $U_{y}$ of $X$ such that $U_{y}\cap f(Y)=\{f(y)\}$. Let $y,z\in Y$ and
$y\neq z$. Then $U_{y}\cap U_{z}=\emptyset$ (since $f(Y)$ is dense in $X$).
This shows that $d(X)\geq |Y|$. Therefore, $d(X)=|Y|$.
\end{proof}
\begin{corollary}
Let $(f_{1},Y_{1})$ and $(f_{2},Y_{2})$ be a pair of discretizations of a topological space $X$.
Then $|Y_{1}|=|Y_{2}|$, that is, $Y_{1}$ is homeomorphic to $Y_{2}$.
\end{corollary}

Let $X$ be a set with at least two elements equipped with the anti-discrete topology. Then each element of $\mathcal{D}(X)$
is equivalent to $(j,\{x\})$ where $x\in X$. Thus the number of elements of $\mathcal{D}(X)$
equals the cardinality of $X$.

The proof of the following proposition is straightforward.

\begin{proposition}\label{proprod}
Let $\{X_{i}\mid i\in I\}$ be a finite family of topological spaces and
$(f_{i},Y_{i})$ be a discretization of $X_{i}$, $i\in I$.
Then $(\prod_{i\in I} f_{i},\prod_{i\in I} Y_{i})$ is a discretization of
$\prod_{i\in I} X_{i}$.
\end{proposition}

The following theorem states that the notion of discretization introduced here
is a dual of the notion of compactification.
Recall that a Hausdorff topological space is said to be
\emph{extremely disconnected} if the closure of every open subset is open.
A compact extremely disconnected space is called \emph{Stonean} \cite[Definition~III.1.6]{ta79}.

\begin{theorem}\label{thrdual}
Let $X$ be a Tychonoff space. Then:
\begin{itemize}
\item[(1)]
 $\delta(\beta(X))=X$ if and only if $X$ is discrete;
\item[(2)]
$\beta(\delta(X))=X$ if and only if $X$ is
 Stonean and $\alpha$-scattered.
\end{itemize}
\end{theorem}
\begin{proof}
Part (1) follows from Theorem~\ref{thrt1}. For (2),
suppose that $\beta(\delta(X))=X$. Since the Stone-\v{C}ech compactification
of any discrete space is extremely disconnected \cite[Theorem~6.2.7]{en89},
$X$ is Stonean. On the other hand,
$\delta(X)$ is dense in $\beta(\delta(X))=X$. Hence by Theorem~\ref{thrchdisc},
$X$ is $\alpha$-scattered.

Now suppose that $X$ is Stonean and $\alpha$-scattered.
By Theorem~\ref{thrchdisc}, $\delta(X)$ is dense in $X$.
On the other hand, the Stone-\v{C}ech compactification of
every open dense subspace of a Stonean space is the space itself \cite[Corollary~III.1.8]{ta79}.
Therefore, $\beta(\delta(X))=X$.
\end{proof}

Part (1) of Theorem~\ref{thrdual} holds for any compactification of
a Tychonoff space $X$. This follows from Theorem~\ref{thrdiscom} and Theorem~\ref{thrt1}.
However, part~(2) does not hold for arbitrary compactifications
(for example for the one-point compactification): Let $X$ be the one-point compactification of $\N$. Then $X$ is $\alpha$-scattered.
Also $\delta(X)= \N$  and so $\beta(\delta(X))=\beta \N$ which is not homeomorphic to $X$.
Thus, by Theorem~\ref{thrdual}, $X$ is not Stonean.
In fact, this can be seen directly as follows.
$X$ can be identified by $\{0\}\cup\{\frac{1}{n}\mid n\in\N\}$
with relative topology from $\R$. Set $U=\{\frac{1}{2n}\mid  n\in\N\}$.
Then the closure of $U$ is not open in $X$.

\section{Duality}
\noindent
We  give a categorical version of Theorem~\ref{thrdual} in this section.
Let us denote by $\mathbf{Top}$ the category of topological spaces
with continuous maps and by $\mathbf{Disc}$ the category of discrete spaces
with (continuous) maps.
\begin{definition}\label{defdiscfu}
A \emph{preweak discretization functor} is a functor
$F: \mathcal{C}\to  \mathbf{Disc}$ where
$\mathcal{C}$ is a subcategory of  $\mathbf{Top}$, such that
there is a correspondence $\sigma$ which assigns to each $X$ in $\mathcal{C}$
 a (continuous) injective map  $\sigma_{X}: F(X)\to X$ such that for each morphism
$f:X\to Y$ in $\mathcal{C}$
the following diagram commutes:
\[
\xymatrix{
F(X) \ar[d]_{F(f)}\ar[r]^{\sigma_{X}} & X \ar[d]^{f} \\
F(Y)\ar[r]_{\sigma_{Y}}  & Y\, .
}
\]
Such a functor $F$ is called a \emph{weak discretization functor} (respectively, \textit{discretization functor})
if each $\sigma_{X}$ is a homeomorphic embedding with open range (respectively, homeomorphic embedding
with dense range).
\end{definition}

Consider the functor $\mathrm{d}: \mathbf{Top}\to \mathbf{Disc}$ defined
by $\mathrm{d}(X)=X_{d}$ and for each morphism
$f:X\to Y$ in $ \mathbf{Top}$, $\mathrm{d}(f)=f: X_{d}\to Y_{d}$.
Then $\mathrm{d}$ is a preweak discretization functor which is not a
weak discretization functor. Alternatively, let $\mathcal{C}$ be the subcategory of $\mathbf{Top}$  defined as follows.
The objects of $\mathcal{C}$ are the same objects of $\mathbf{Top}$.
A morphism in $\mathcal{C}$ is a continuous map
$f:X\to Y$ with the property that $f(\delta X)\subseteq \delta Y$.
Define the functor $\delta :\mathcal{C}\to  \mathbf{Disc}$ by
$\delta(X)=\delta X$, the isolated points of $X$, and
for each morphism $f:X\to Y$ in $\mathcal{C}$,
$\delta(f)=f\upharpoonright_{\delta X}: \delta X\to \delta Y$.
Then $\delta :\mathcal{C}\to  \mathbf{Disc}$ is a
weak discretization functor which is not a
 discretization functor.
Let $\mathcal{D}$ be the full subcategory of $\mathcal{C}$ whose objects
are  $\alpha$-scattered and Stonean spaces.
Then $\delta :\mathcal{D}\to  \mathbf{Disc}$ is a
discretization functor.

The following duality theorem provides an inverse for the Stone-\v{C}ech compactification functor
on the category of discrete spaces.

\begin{theorem}\label{thrcdual}
Let $\delta :\mathcal{D}\to  \mathbf{Disc}$ be the discretization functor defined  above and
$\beta :\mathbf{Disc}\to \mathcal{D}$ be the restriction of the Stone-\v{C}ech compactification functor
to $\mathbf{Disc}$.
Then $\delta \beta\cong\mathrm{id}_{\mathbf{Disc}}$ and
$\beta\delta\cong\mathrm{id}_{\mathcal{D}}$, that is,
$\beta$ and $\delta$ are equivalences of categories and natural inverse of each other.
\end{theorem}
\begin{proof}
First let us show that $\delta \beta\cong\mathrm{id}_{\mathbf{Disc}}$.
Let $X$ be in $\mathbf{Disc}$. Consider the embedding
$\beta_{X}:X\to \beta X$ as described in Section~2. By Theorem~\ref{thrdual},
$\beta_{X}(X)=\delta(\beta(X))$. Consider the homeomorphism
$\xi_{X}:X\to \delta(\beta(X))$ which is the map $\beta_{X}$ onto its range.
It is easily verified that
$\xi: \mathrm{id}_{\mathbf{Disc}}\cong \delta \beta$.

To show that $\beta\delta\cong\mathrm{id}_{\mathcal{D}}$ let
$X$ be in $\mathcal{D}$, i.e.,
$X$ is a Stonean $\alpha$-scattered space.
By \cite[Corollary~III.1.8]{ta79},
$\beta(\delta(X))$ is homeomorphic to $X$. More precisely,
there is a homeomorphism
$\eta_{X}:X\to \beta(\delta(X))$ which is
the continuous extension  of
the injection map $\delta(X)\hookrightarrow X$.
One can easily check that
$\eta : \mathrm{id}_{\mathcal{D}}\cong \beta\delta$.
\end{proof}
\begin{corollary}
The categories $\mathbf{Disc}$ and $\mathcal{D}$ are equivalent.
\end{corollary}

\section{Discretization in Function Algebras}
\noindent
In Section~3, we saw that various compactifications of a Tychonoff
space $X$ correspond to certain unital subalgebras of $\mathrm{C}_{b}(X)$.
In this section we show that various discretizations of a locally compact Hausdorff
space $X$ correspond to certain ideals of $\mathrm{C}_{0}(X)$.

Recall that for a  locally compact Hausdorff space $X$, there is one-to-one correspondence between open subsets of
 $X$ and closed ideals of $\mathrm{C}_{0}(X)$ \cite[Exercise~2.A]{wo93}.
 In fact, let $Tp(X)$ be the set of all open subsets of $X$ (i.e., the topology of $X$),
 and $\mathrm{Id}(\mathrm{C}_{0}(X))$ be the set of all closed (in uniform metric)
 ideals of the algebra $\mathrm{C}_{0}(X)$.
 Define the mapping
 $I:T(X)\to \mathrm{Id}(\mathrm{C}_{0}(X))$ as follows.
 For each open subset $U$ of $X$ set

 \[
 I(U)=\{f\in \mathrm{C}_{0}(X) : f=0\  \text{ on}\  X\setminus U\}.
\]
 Then $I$ is bijective and for each closed ideal $J$ of $\mathrm{C}_{0}(X)$,
 the inverse image of $J$ is $\{x\in X : f(x)\neq 0\ \text{for some}\ f\in J\}$.
In fact, $I$ is an order isomorphism between the complete lattices $(Tp(X),\subseteq)$
and $(\mathrm{Id}(\mathrm{C}_{0}(X)), \subseteq)$.
For each open subset
 $U$ of $X$, the map
 $\varphi_{U}:\mathrm{C}_{0}(U)\to I(U)$, which is defined by $\varphi_{U}(f)=f$ on $U$
 and $\varphi_{U}(f)=0$ on $X\setminus U$, is a $*$-isomorphism of C$^*$-algebras.

 Now let $(f,Y)$ be a weak discretization of a  locally compact Hausdorff space $X$.
 Define the closed ideal $I(f,Y)$ as follows.
The homeomorphism $f:Y\to f(Y)$ induces the $*$-isomorphism
$\mathrm{C}_{0}(f):\mathrm{C}_{0}(f(Y))\to \mathrm{C}_{0}(Y)$.
Consider the injective $*$-homomorphism
$\varphi_{f(Y)}\mathrm{C}_{0}(f)^{-1}:\mathrm{C}_{0}(Y) \to  \mathrm{C}_{0}(X)$.
The range of this map is denoted by $I(f,Y)$.
Note that with the notation of the preceding paragraph, we have
\[
I(f,Y)=I(f(Y))=\{g\in \mathrm{C}_{0}(X) : g=0\  \text{ on}\  X\setminus f(Y)\}.
\]
Therefore, we have the map $I: \mathbf{Disc}_{w}(X)\to \mathrm{Id}( \mathrm{C}_{0}(X))$.
Using this map, we can obtain an algebraic counterpart of discretization.
\begin{lemma}\label{lemdisc}
Let $X$ be a locally compact Hausdorff space and
$(f,Y)$ and $(g,Z)$ be a pair of weak discretizations of $X$.
Then $(f,Y)\leq (g,Z)$ if and only if $I(f,Y)\subseteq I(g,Z)$.
\end{lemma}
\begin{proof}
Suppose that $(f,Y)\leq (g,Z)$. Then there is a map
$h:Y\to Z$ such that $gh=f$.
Thus $f(Y)\subseteq g(Z)$. This implies that $I(f,Y)\subseteq I(g,Z)$.

Conversely, suppose that $I(f,Y)\subseteq I(g,Z)$.
Then $f(Y)\subseteq g(Z)$.
Since $f$ and $g$ are injective,
 there is a unique map $h:Y\to Z$ such that $gh=f$.
 Hence $(f,Y)\leq (g,Z)$.
\end{proof}
The above lemma shows that
$(f,Y)$ and $(g,Z)$ are equivalent if and only if $I(f,Y)=I(g,Z)$. Therefore,
the map $I: \mathbf{Disc}_{w}(X)\to \mathrm{Id}( \mathrm{C}_{0}(X))$ induces
the injective order preserving map $I: \mathcal{D}_{w}(X)\to \mathrm{Id}( \mathrm{C}_{0}(X))$.
We need to find the range of this map to complete the correspondence.

Let us recall some definitions from operator algebras \cite{mu90, ta79}. Let
$X$ be a locally compact Hausdorff space.
An element $p$ in $\mathrm{C}_{0}(X)$ is called a \emph{projection} if
$p=p^{*}=p^{2}$, that is, $p$ is the characteristic function of a clopen subset of $X$.
Let $p$ and $q$ be a pair of projections in $\mathrm{C}_{0}(X)$.
We write $p\leq q$ if the range of $p$ is contained in the range of $q$.
Thus $\leq$ is a partial ordering on the set of all projections.
A \textit{minimal  projection} in $\mathrm{C}_{0}(X)$ is a non-zero projection
which is minimal in the partially ordered set of all projections.
Recall the definition of a C$^{*}$-subalgebra of $\mathrm{C}_{0}(X)$ from
Section~3. The C$^{*}$-subalgebra generated by a subset of
$\mathrm{C}_{0}(X)$ is the smallest C$^{*}$-subalgebra containing that set.
The following result is well known \cite{bl94}.
\begin{theorem}\label{thrzero}
Let $X$ be a locally compact Hausdorff space.  Then
$X$ is zero-dimensional if and only if
$\mathrm{C}_{0}(X)$ is generated by its projections.
\end{theorem}
\begin{theorem}\label{thrdis}
Let $X$ be a locally compact Hausdorff space. Then
$X$ is discrete if and only if $\mathrm{C}_{0}(X)$ is generated by its minimal
projections.
\end{theorem}
\begin{proof}
If $X$ is discrete then $\mathrm{C}_{0}(X)$ is generated by $\chi_{\{x\}}$, $x\in X$,
and these are minimal projections of  $\mathrm{C}_{0}(X)$. For the converse,
suppose that $\mathrm{C}_{0}(X)$ is generated by its minimal
projections. By Theorem~\ref{thrzero}, $X$ is
zero-dimensional.
 Let $\chi_{\{A\}}$ be a minimal  projection in
$\mathrm{C}_{0}(X)$ for some $A\subseteq X$. We claim that $A$ is a singleton.
Suppose otherwise that $a,b\in A$ and $a\neq b$. Then there is a compact clopen subset
$U$ of $X$ such that $a\in U$ and $b\not\in U$. Hence
$\chi_{\{U\}}$ is a  projection in $\mathrm{C}_{0}(X)$ and
$\chi_{U}<\chi_{A}$
which is a contradiction. Therefore $A$ is a singleton.

Now let $x\in X$; we show that $\{x\}$ is open in $X$.
There is  $f$ in $\mathrm{C}_{0}(X)$ such that $f(x)=1$. Since $\mathrm{C}_{0}(X)$
is generated by its minimal  projections, there is a sequence
$\{f_{n}\}_{n\geq 1}$ of linear span of projections of the form
$\chi_{\{a\}}$, $a\in X$, such that $f_{n}$ tend to $f$ uniformly on $X$.
Thus $f_{n}(x)\to f(x)=1$ and hence $f_{n}(x)\neq 0$, for some $n$.
This shows that $\{x\}$ is open in $X$. Therefore, $X$ is discrete.
\end{proof}
Let us denote by $\mathrm{Id_{gmp}}( \mathrm{C}_{0}(X))$ the set of
all closed  ideals of $ \mathrm{C}_{0}(X)$ which are generated (as
C$^{*}$-subalgebras) by their minimal projections. (Note that the zero ideal
is in $\mathrm{Id_{gmp}}( \mathrm{C}_{0}(X))$.)
\begin{proposition}\label{propgmp}
Let $X$ be a locally compact Hausdorff space. Then the map
$I: \mathcal{D}_{w}(X)\to \mathrm{Id_{gmp}}(\mathrm{C}_{0}(X))$
is an isomorphism of complete lattices.
\end{proposition}
\begin{proof}
We have shown that the map $I: \mathcal{D}_{w}(X)\to \mathrm{Id}( \mathrm{C}_{0}(X))$
is  injective and order-preserving. By Proposition~\ref{propwdisc},
$\mathcal{D}_{w}(X)$ is a complete lattice. Thus we need only to show that the range of
this map is $\mathrm{Id_{gmp}}(\mathrm{C}_{0}(X))$.
Let $(f,Y)$ be a discretization of $X$. Since $\mathrm{C}_{0}(f(Y))$
is $*$-isomorphic to $I(f,Y)$ (via the map $\varphi_{f(Y)}$),
$I(f,Y)$ is generated by its minimal projections, by Theorem~\ref{thrdis}.
Thus the range of $I$ is contained in $\mathrm{Id_{mpr}}(\mathrm{C}_{0}(X))$.

Now let $J$ be in $\mathrm{Id_{mpr}}(\mathrm{C}_{0}(X))$.
Then there is an open subset $U$ of $X$ such that $J$ is the range of
$\varphi_{U}: \mathrm{C}_{0}(U)\to \mathrm{C}_{0}(X)$.
Since $J$ is generated by its minimal projections, so is
$\mathrm{C}_{0}(U)$. By Theorem~\ref{thrdis}, $U$ is discrete.
Thus $(j,U)$, where $j$ is the inclusion map from $U$ to $X$, is a weak discretization of $X$ and $I$ maps
$(j,U)$ to $J$.
\end{proof}
Recall that the map  $I:Tp(X)\to \mathrm{Id}(\mathrm{C}_{0}(X))$, defined above,
maps open dense subspaces of $X$ to closed essential ideals of
$\mathrm{C}_{0}(X)$ \cite[Exercise~2.A]{wo93}. Recall that
a closed ideal $J$ of $\mathrm{C}_{0}(X)$ is called \textit{essential} if
$J\cap K\neq \{0\}$ for every non-zero closed ideal $K$ of $\mathrm{C}_{0}(X)$.
Let us denote by
$\mathrm{Id_{gmp}^{ess}}(\mathrm{C}_{0}(X))$ the essential ideals in
$\mathrm{Id_{gmp}}(\mathrm{C}_{0}(X))$.
The map $I: \mathbf{Disc}_{w}(X)\to \mathrm{Id}( \mathrm{C}_{0}(X))$
induces the  map $I: \mathcal{D}(X)\to \mathrm{Id}( \mathrm{C}_{0}(X))$,
by Lemma~\ref{lemdisc}.
\begin{proposition}\label{propess}
Let $X$ be a locally compact Hausdorff space. Then the map
$I: \mathcal{D}(X)\to \mathrm{Id_{gmp}^{ess}}(\mathrm{C}_{0}(X))$
is bijective.
\end{proposition}
\begin{proof}
The statement follows from Lemma~\ref{lemdisc}, Proposition~\ref{propgmp}, and  the above remarks.
\end{proof}
The above proposition gives an algebraic characterization of discretization.
\begin{corollary}\label{cordisc}
Let $X$ be a locally compact Hausdorff space. Then the following are equivalent:
\begin{itemize}
\item[(1)]
$X$ has a discretization;
\item[(2)]
$\mathrm{C}_{0}(X)$ has a (unique) closed essential ideal which is generated by
its minimal projections.
\end{itemize}
\end{corollary}
\begin{corollary}
Let $X$ be a locally compact Hausdorff space. Then the set
$\mathrm{Id_{gmp}^{ess}}(\mathrm{C}_{0}(X))$ has at most one element.
\end{corollary}
\begin{proof}
The statement follows from  Propositions~\ref{propwdisc} and \ref{propess}.
\end{proof}


\begin{thebibliography}{99}
\bibitem{ac95} A.\,V. Arhangel'skii, P.\,J. Collins, \emph{On submaximal spaces}, Topology Appl. \textbf{64}
(1995), no. 3, 219--241.
\bibitem{bl94} B. Blackadar, \emph{Projections in C$^{*}$-algebras}, Contemporary Mathematics, vol. 167, 1994.
\bibitem{co90} John B. Conway, \emph{A Course in Functional Analysis}, Vol. 96, Springer, 1990.
\bibitem{bjm89} J.\,F. Berglund, H.\,D. Junghenn, P. Milnes,
\emph{Analysis on Semigroups: Function Spaces, Compactifications, Representations},
John Wiley \& Sons, Inc., 1989.
\bibitem{bcm00} M. Bonanzinga, F. Cammaroto, M.\,V. Matveev, \emph{Partial discretization of topologies},
 Boll. Unione Mat. Ital. Sez. B Artic. Ric. Mat. (8) 3 (2000), no. 2, 485--503.
\bibitem{bs81} S. Buris, H.\,P. Sankappanavar, \emph{A Course in Universal Algebra}, Springer, 1981.
\bibitem{en89} R. Engelking, \emph{General Topology}, Heldermann
Verlag, Berlin, 1989.
\bibitem{fo95} G.\,B. Folland,  \emph{A Course in Abstract Harmonic Analysis}, CRC Press, Boca Raton, FL, 1995.
\bibitem{ho64} P. Holm, \emph{On the Bohr compactification}, Math. Annalen \textbf{156} (1964), 34--46.
\bibitem{ke75} J. Kelley, \emph{General Topology}, Springer, New York, 1975.
\bibitem{ma98} S. Mac Lane, \emph{Categories for the  Working Mathematician},
Second edition, Springer, New York, 1998.
\bibitem{mu90} G. Murphy, \emph{C$^{*}$-Algebras and Operator Theory},
 Academic Press, New York, 1990.
\bibitem{ro98} D.\,A. Rose, \emph{$\alpha$-scattered spaces}, Internat. J. Math.   Math. Sci., vol. 21, no. 1
 (1998), 41--46.
\bibitem{ta79} M. Takesaki, \emph{Theory of Operator Algebras, vol. I}, Springer, New York, 1979.
\bibitem{wo93} N.\,E. Wegge-Olsen, \emph{K-Theory and
C$^{*}$-Algebras}, The Clarendon Press, New York, 1993.
\bibitem{wi70} S. Willard, \emph{General Topology}, Addison-Wesley, 1970.
 \end{thebibliography}
\end{document}